\documentclass[a4paper,11pt]{amsart}
\usepackage{graphicx}
 \usepackage{mathptmx}
\usepackage{amsmath}
\usepackage{amssymb}
\usepackage{enumitem}
\usepackage{xcolor}

\newmuskip\pFqmuskip

\newcommand*\pFq[6][8]{%
  \begingroup 
  \pFqmuskip=#1mu\relax
  \mathcode`=\string"8000
  \begingroup\lccode`\~=`\,
  \lowercase{\endgroup\let~}\pFqcomma
  F^{#2}_{#3}{\left(\genfrac..{0pt}{}{#4}{#5}\bigg|#6\right)}%
  \endgroup
}
\newcommand{\pFqcomma}{\mskip\pFqmuskip}

\newtheorem{theorem}{Theorem}[section]

\newtheorem{remark}[theorem]{Remark}

\begin{document}

\title[A note on Leibniz rule for difference quotient]{A note on Leibniz rule for difference quotient}

\author{Taekyun Kim}
\address{Department of Mathematics, Kwangwoon University, Seoul 139-701, Republic of Korea}
\email{tkkim@kw.ac.kr}
\author{Dae San Kim}
\address{Department of Mathematics, Sogang University, Seoul 121-742, Republic of Korea}
\email{dskim@sogang.ac.kr}

\subjclass[2010]{65Q10}
\keywords{Leibniz rule; difference quotient}

\maketitle

\begin{abstract}
In this note, we derive a Leibniz rule for difference quotient.
\end{abstract}

\section{Leibniz rule for difference quotient}

The Finite Difference Method (FDM) is a numerical technique used to solve differential equations by approximating continuous derivatives with discrete algebraic equations. The process begins by dividing the continuous domain (like a string or a metal plate) into a grid of discrete points (nodes) separated by a fixed distance, $\Delta x$ (or $h$). Instead of finding a smooth function, we solve for the values at these specific points. FDM replaces derivatives with difference quotients derived from Taylor series expansions. These formulas use the values at neighboring grid points to estimate the slope or curvature. By substituting these quotients into a differential equation, the problem is transformed into a system of linear equations (see [8]). This note is written with certain applications to degenerate versions of some special polynomials and numbers in mind. For this, the reader may refer to some of the recent papers on this topic (see [1-7]).\par
In this note, we prove a Leibniz rule for two functions in Theorem 1.1 and then extend it to any number of functions in Theorem 1.3.
We denote the difference quotient of the function $f$ by
\begin{equation}\label{1}
(\delta_{\lambda}f)(x)=\frac{f(x+\lambda)-f(x)}{\lambda},
\end{equation}
where $\lambda$ is any nonzero real number and $f$ is any complex-valued function defined on the real line.
We note that $\delta_{\lambda}$ is linear. First, we derive the product rule for the difference quotient in \eqref{1}:

\begin{align}
\delta_{\lambda}(fg)(x)&=\frac{1}{\lambda}\big(f(x+\lambda)g(x+\lambda)-f(x)g(x)\big) \label{2}\\
&=\frac{1}{\lambda}\big(f(x+\lambda)g(x+\lambda)-f(x)g(x+\lambda)+f(x)g(x+\lambda)-f(x)g(x)\big) \nonumber\\
&=\frac{1}{\lambda}\big(f(x+\lambda)-f(x)\big)g(x+\lambda)+f(x)\big(g(x+\lambda)-g(x)\big) \nonumber\\
&=g(x+\lambda)\bigg(\frac{f(x+\lambda)-f(x)}{\lambda}\bigg)+f(x)\bigg(\frac{g(x+\lambda)-g(x)}{\lambda}\bigg) \nonumber\\
&=g(x+\lambda)(\delta_{\lambda}f)(x)+f(x)(\delta_{\lambda}g)(x) \nonumber \\
&=(\delta_{\lambda}f)(x)(\lambda (\delta_{\lambda}g)(x)+g(x))+f(x)(\delta_{\lambda}g)(x)\nonumber\\
&=\lambda (\delta_{\lambda}f)(x)(\delta_{\lambda}g)(x)+(\delta_{\lambda}f)(x)g(x)+f(x)(\delta_{\lambda}g)(x). \nonumber
\end{align}
The product rule for difference quotient in \eqref{2} is stated as:
\begin{equation}\label{3}
\delta_{\lambda}(fg)=\lambda(\delta_{\lambda} f)(\delta_{\lambda} g)+(\delta_{\lambda} f)g+f(\delta_{\lambda} g).
\end{equation}
We prove the following Leibniz rule for difference quotient by induction. 
\begin{theorem}
For any positive integer $r$, we have the Leibniz rule for difference quotient given by
\begin{equation} \label{4}
\delta_{\lambda}^{r}(fg)=\sum_{l=0}^{r}\lambda^{r-l}\binom{r}{l}\sum_{k=0}^{l}\binom{l}{k}(\delta_{\lambda}^{r-k}f)(\delta_{\lambda}^{k+r-l}g).
\end{equation}
\begin{proof}
In the course of this proof, we denote $\delta_{\lambda}$ simply by $\delta$. For $r=1$, this is just the formula in \eqref{3}. Assume that \eqref{4} holds for $r$. 
Applying the formula \eqref{3} to $(\delta^{r-k}f)(\delta^{k+r-l}g)$ yields:
\begin{align}
\delta\big( (\delta^{r-k}f)(\delta^{k+r-l}g) \big)&=\lambda (\delta^{r-k+1}f)(\delta^{k+r-l+1}g)  \label{5} \\
&\quad+(\delta^{r-k+1}f)(\delta^{k+r-l}g)+(\delta^{r-k}f)(\delta^{k+r-l+1}g). \nonumber
\end{align}
Using the linearity of  $\delta$ and \eqref{5}, applying $\delta$ to \eqref{4} results in the following sums:
\begin{align}
\delta^{r+1}(fg)&=\sum_{l=0}^{r}\lambda^{r-l+1}\binom{r}{l}\sum_{k=0}^{l}\binom{l}{k}(\delta^{r-k+1}f)(\delta^{k+r-l+1}g) \nonumber \\
&\quad+\sum_{l=0}^{r}\lambda^{r-l}\binom{r}{l}\sum_{k=0}^{l}\binom{l}{k}(\delta^{r-k+1}f)(\delta^{k+r-l}g) \label{6} \\
&\quad+\sum_{l=0}^{r}\lambda^{r-l}\binom{r}{l}\sum_{k=0}^{l}\binom{l}{k}(\delta^{r-k}f)(\delta^{k+r-l+1}g). \nonumber
\end{align} 
Now, we note that 
\begin{align}
&\sum_{k=0}^{l}\binom{l}{k}(\delta^{r-k+1}f)(\delta^{k+r-l}g)+\sum_{k=0}^{l}\binom{l}{k}(\delta^{r-k}f)(\delta^{k+r-l+1}g) \nonumber\\
&=\sum_{k=0}^{l}\binom{l}{k}(\delta^{r-k+1}f)(\delta^{k+r-l}g)+\sum_{k=1}^{l+1}\binom{l}{k-1}(\delta^{r-k+1}f)(\delta^{k+r-l}g) \nonumber\\
&=(\delta^{r+1}f)(\delta^{r-l}g)+(\delta^{r-l}f)(\delta^{r+1}g) \label{7}\\
&\quad +\sum_{k=1}^{l}\Big\{\binom{l}{k}+\binom{l}{k-1}\Big\}(\delta^{r-k+1}f)(\delta^{k+r-l}g) \nonumber\\
&=\sum_{k=0}^{l+1}\binom{l+1}{k}(\delta^{r-k+1}f)(\delta^{k+r-l}g).\nonumber
\end{align}
By \eqref{6} and \eqref{7}, after rearranging the terms, we see that
\begin{align*}
\delta^{r+1}(fg)&=\sum_{l=0}^{r}\lambda^{r-l+1}\binom{r}{l}\sum_{k=0}^{l}\binom{l}{k}(\delta^{r-k+1}f)(\delta^{k+r-l+1}g) \\
&\quad+\sum_{l=0}^{r}\lambda^{r-l}\binom{r}{l}\sum_{k=0}^{l+1}\binom{l+1}{k}(\delta^{r-k+1}f)(\delta^{k+r-l}g) \\
&=\sum_{l=0}^{r}\lambda^{r-l+1}\binom{r}{l}\sum_{k=0}^{l}\binom{l}{k}(\delta^{r-k+1}f)(\delta^{k+r-l+1}g)  \\
&\quad+\sum_{l=1}^{r+1}\lambda^{r-l+1}\binom{r}{l-1}\sum_{k=0}^{l}\binom{l}{k}(\delta^{r-k+1}f)(\delta^{k+r-l+1}g) \\
&=\lambda^{r+1}(\delta^{r+1}f)(\delta^{r+1}g)+\sum_{k=0}^{r+1}\binom{r+1}{k}(\delta^{r-k+1}f)(\delta^{k}g) \\
&\quad+\sum_{l=1}^{r}\lambda^{r-l+1}\Big\{\binom{r}{l}+\binom{r}{l-1}\Big\}\sum_{k=0}^{l}\binom{l}{k}(\delta^{r-k+1}f)(\delta^{k+r-l+1}g) \\
&=\sum_{l=0}^{r+1}\lambda^{r-l+1}\binom{r+1}{l}\sum_{k=0}^{l}\binom{l}{k}(\delta^{r-k+1}f)(\delta^{k+r-l+1}g),
\end{align*}
which shows the validity of \eqref{4} for $r+1$.
\end{proof}
\end{theorem}
\begin{remark}
Taking $\lambda \rightarrow 0$, we recover the classical Leibniz rule:
\begin{equation*}
\Big(\frac{d}{dx}\Big)^{r}(f(x)g(x))=\sum_{k=0}^{r}\binom{r}{k}f^{(r-k)}(x)g^{(k)}(x).
\end{equation*}
\end{remark}

To extend the rule Theorem 1.1 for more than two functions, we introduce the operator:
\begin{equation}
L_{\lambda}=I+\lambda \delta_{\lambda}. \label{8}
\end{equation}
Multiplying both sides of \eqref{3} by $\lambda$ and then adding $fg$ to the resulting equation yield:
\begin{align*}
L_{\lambda}(fg)=fg+\lambda \delta_{\lambda}(fg)&=fg+\lambda^{2}(\delta_{\lambda} f)(\delta_{\lambda} g)+\lambda(\delta_{\lambda} f)g+\lambda f(\delta_{\lambda} g) \\
&=(f+\lambda \delta_{\lambda}f)(g+\lambda \delta_{\lambda}g)=L_{\lambda}(f)L_{\lambda}(g).
\end{align*}
This shows that the equation \eqref{3} is equivalent to saying the operator $L_{\lambda}$ is multiplicative on the algebra of complex-valued functions on the real line. Then, by induction, it is immediate to see that, for any positive integers $r$ and $n$,
\begin{equation}
L_{\lambda}^{r}(f_{1}f_{2}\cdots f_{n})=L_{\lambda}^{r}(f_{1})L_{\lambda}^{r}(f_{2})\cdots L_{\lambda}^{r}(f_{n}). \label{9}
\end{equation}
Noting that $\delta_{\lambda}=\frac{L_{\lambda}-I}{\lambda}$ and using \eqref{9}, we have:
\begin{align}
\delta_{\lambda}^{r}(f_{1}f_{2}\cdots f_{n})&=\frac{1}{\lambda^{r}}(L_{\lambda}-I)^{r}(f_{1}f_{2}\cdots f_{n}) \nonumber \\
&=\frac{1}{\lambda^{r}}\sum_{k=0}^{r}\binom{r}{k}(-1)^{r-k}L_{\lambda}^{k}(f_{1}f_{2}\cdots f_{n}) \label{10} \\
&=\frac{1}{\lambda^{r}}\sum_{k=0}^{r}\binom{r}{k}(-1)^{r-k}L_{\lambda}^{k}(f_{1})L_{\lambda}^{k}(f_{2})\cdots L_{\lambda}^{k}(f_{n}), \nonumber
\end{align}
where, for any $i=1,2, \dots, n$,
\begin{equation*}
L_{\lambda}^{k}(f_{i})=(I+\lambda \delta_{\lambda})^{k}(f_{i})=\sum_{j_{i}=0}^{k}\binom{k}{j_{i}}\lambda^{j_{i}}\delta_{\lambda}^{j_{i}}f_{i}.
\end{equation*}
For $r=1$ and $n=3$, \eqref{10} gives:
\begin{align*}
\delta_{\lambda}(f_{1}f_{2}f_{3})&=\frac{1}{\lambda}\sum_{k=0}^{1}\binom{1}{k}(-1)^{1-k}L_{\lambda}^{k}(f_{1})L_{\lambda}^{k}(f_{2})L_{\lambda}^{k}(f_{3}) \\
&=\frac{1}{\lambda}\Big[(f_{1}+\lambda \delta_{\lambda}f_{1})(f_{2}+\lambda \delta_{\lambda}f_{2})(f_{3}+\lambda \delta_{\lambda}f_{3})-f_{1}f_{2}f_{3}\Big].
\end{align*}
For $r=2$ and $n=3$, \eqref{10} yields:
\begin{align*}
\delta_{\lambda}^{2}(f_{1}f_{2}f_{3})&=\frac{1}{\lambda^{2}}\sum_{k=0}^{2}\binom{2}{k}(-1)^{2-k}L_{\lambda}^{k}(f_{1})L_{\lambda}^{k}(f_{2})L_{\lambda}^{k}(f_{3})\\
&=\frac{1}{\lambda^{2}}\Big[L_{\lambda}^{2}(f_{1})L_{\lambda}^{2}(f_{2})L_{\lambda}^{2}(f_{3})-2L_{\lambda}(f_{1})L_{\lambda}(f_{2})L_{\lambda}(f_{3})+f_{1}f_{2}f_{3}\Big] \\
&= \frac{1}{\lambda^{2}} \Big[ \prod_{i=1}^3 (f_{i }+ 2\lambda \delta_{\lambda}f_{i} + \lambda^2 \delta_{\lambda}^{2} f_{i}) - 2 \prod_{i=1}^3 (f_{i }+ \lambda \delta_{\lambda}f_{i}) + f_{1} f_{2} f_{3} \Big].
\end{align*}
We state the result in \eqref{10} as a theorem.
\begin{theorem}
For any positive integers $r$ and $n$, we have the Leibniz rule for the different quotient given by
\begin{equation}
\delta_{\lambda}^{r}(f_{1}f_{2}\cdots f_{n})=\frac{1}{\lambda^{r}}\sum_{k=0}^{r}\binom{r}{k}(-1)^{r-k}L_{\lambda}^{k}(f_{1})L_{\lambda}^{k}(f_{2})\cdots L_{\lambda}^{k}(f_{n}), \label{11}
\end{equation}
\end{theorem}
\noindent where, for any $i=1,2, \dots, n$,
\begin{equation*}
L_{\lambda}^{k}(f_{i})=(I+\lambda \delta_{\lambda})^{k}(f_{i})=\sum_{j_{i}=0}^{k}\binom{k}{j_{i}}\lambda^{j_{i}}\delta_{\lambda}^{j_{i}}f_{i}. 
\end{equation*}

\begin{remark}
(a) The presence of minus signs in \eqref{11} tells us that there are many cancellations. \\
(b) We write \eqref{11} as
\begin{equation}
\lambda^{r} \delta_{\lambda}^{r}(f_{1}f_{2}\cdots f_{n})=\sum_{k=0}^{r}\binom{r}{k}(-1)^{r-k}L_{\lambda}^{k}(f_{1}f_{2}\cdots f_{n}). \label{12}
\end{equation}
Then, by binomial inversion, we also have
\begin{equation}
L_{\lambda}^{r}(f_{1}f_{2}\cdots f_{n})=\sum_{k=0}^{r}\binom{r}{k}\lambda^{k} \delta_{\lambda}^{k}(f_{1}f_{2}\cdots f_{n}). \label{13}
\end{equation}
For each positive integer $n$, we let
\begin{equation*}
A_{n}(t)=\sum_{r=0}^{\infty}\lambda^{r}\delta_{\lambda}^{r}(f_{1}f_{2}\cdots f_{n})\frac{t^{r}}{r!}, \quad \bar{A}_{n}(t)=\sum_{r=0}^{\infty} L^{r}(f_{1}f_{2} \cdots f_{n})\frac{t^{r}}{r!}.
\end{equation*}
Then, as we can check directly or by the basic fact on Euler-Seidel matrix, the two binomial relations in \eqref{12} and \eqref{13} can be concisely expressed as follows:
\begin{equation*}
\bar{A}_{n}(t)=e^{t}A_{n}(t).
\end{equation*}
\end{remark}


\begin{thebibliography}{9} 
\bibitem{1}
Kim, D.S.; Kim, T. \emph{Combinatorial identities related to degenerate Stirling numbers of the second Kind,} Proc. Steklov Inst. Math. \textbf{330} (2025), 176-192. https://doi.org/10.1134/S0081543825600279
\bibitem{2}
Kim, D. S.; Kim, T. \emph{Degenerate Sheffer sequences and $\lambda$-Sheffer sequences,} J. Math. Anal. Appl. \textbf{493} (2021), no. 1, Paper No. 124521, 21 pp.
\bibitem{3}
Kim, T.; Kim, D. S. \emph{Degenerate Bernstein polynomials,} Rev. R. Acad. Cienc. Exactas Fs. Nat. Ser. A Mat. RACSAM 113 (2019), no. 3, 2913-2920.
\bibitem{4}
Kim, T.; Kim, D. S. \emph{Spivey-type recurrence relations for degenerate Bell and Dowling polynomials,} Russ. J. Math. Phys. \textbf{32} (2025), no. 2, 288-296.
\bibitem{5} 
Kim, T.; Kim, D. S. \emph{Heterogeneous Stirling numbers and heterogeneous Bell polynomials,} Russ. J. Math. Phys. \textbf{32} (2025), no. 3, 498-509.
\bibitem{6}
Kim, T.; Kim, D. S. \emph{Recurrence relations for degenerate Bell and Dowling polynomials via Boson operators,} Comput. Math. Math. Phys. \textbf{65} (2025), no. 9, 2087-2096.
\bibitem{7}
Kim, W. J.; Kim, D. S.; Kim, H. Y.; Kim, T. \emph{Some identities of degenerate Euler polynomials associated with degenerate Bernstein polynomials,} J. Inequal. Appl. 2019, Paper No. 160, 11pp.
\bibitem{8}
LeVeque, R. J. \emph{Finite difference methods for ordinary and partial differential equations: steady-state and time-dependent problems,} Society for Industrial and Applied Mathematics (SIAM), Philadelphia, 2007.
\end{thebibliography}
\end{document}